\documentclass[12pt]{amsart}
\usepackage{amssymb}
\input amssym.def
\usepackage{amsmath,amsfonts,hyperref,xcolor}
\usepackage{amscd}
\usepackage[mathscr]{eucal}
\hypersetup{colorlinks=true,citecolor={purple},linkcolor={teal},urlcolor={violet}}

\newcommand{\N}{{\mathbb N}}
\newcommand{\Z}{{\mathbb Z}}
\newcommand{\Q}{{\mathbb Q}}
\newcommand{\C}{{\mathbb C}}

\newcommand{\z}{\zeta}

\newtheorem{thm}{Theorem}

\newtheorem{cor}{Corollary}
\newtheorem{prop}{Proposition}
\newtheorem*{ques}{Question}
\newtheorem{rmk}{Remark}
\newtheorem{defn}{Definition}
\newtheorem{ex}{Example}
\newcommand{\thmref}[1]{Theorem~\ref{#1}}
\newcommand{\propref}[1]{Proposition~\ref{#1}}

\newcommand{\corref}[1]{Corollary~\ref{#1}}
\newcommand{\rmkref}[1]{Remark~\ref{#1}}

\begin{document}

\title{Multiple Dirichlet series associated to additive and Dirichlet
characters}

\author{Biswajyoti Saha}

\address{Biswajyoti Saha\\ \newline
Institute of Mathematical Sciences, C.I.T. Campus, Taramani, 
Chennai, 600 113, India.}
\email{biswajyoti@imsc.res.in}



\begin{abstract}
In this article we study analytic properties of the multiple
Dirichlet series associated to additive and Dirichlet characters.
For the multiple Dirichlet series associated to additive characters,
the meromorphic continuation is established via obtaining
translation formulas satisfied by these
multiple Dirichlet series. While it seems difficult to obtain such
a translation formula for the multiple Dirichlet series
associated to Dirichlet characters, we rely on their intrinsic
connection with the multiple Dirichlet series associated to additive
characters in order to investigate their analytic characteristics. We are
also able to determine the exact set of singularities of the
multiple Dirichlet series associated to additive characters.
\end{abstract}

\maketitle

\section{Introduction}

Perhaps the most classical and the most studied multiple Dirichlet series
are the so-called multiple zeta functions. For an integer $r \ge 1$,
consider the open subset $U_r$ of $\C^r$: 
$$
U_r := \{ (s_1, \ldots, s_r) \in \C^r ~:~ \Re(s_1 + \cdots + s_i) > i
~\text{ for all }~ 1 \le i  \le r \}.
$$
The multiple zeta function of depth $r$, denoted by
$\zeta_r(s_1,\ldots,s_r)$, is a function on $U_r$
defined by 
$$
\zeta_r(s_1,\ldots,s_r):=\sum_{n_1>\cdots>n_r>0} 
n_1^{-s_1} \cdots n_r^{-s_r}.
$$
The above series converges normally on any compact subset 
of $U_r$ (see \cite{MSV} for a proof) and hence defines an
analytic function on $U_r$. This is a multi-variable generalisation
of the classical Riemann zeta function. The origin of these functions
can be traced back to Euler, who studied the case $r=2$.

The analytic theory of the multiple zeta functions has been in
the focus of research in recent years. As a function of several variables,
the meromorphic continuation of the multiple zeta function of depth $r$
was first established by J. Zhao \cite{JZ} in 1999. He used the theory
of generalised functions.

\begin{thm}[Zhao]
The multiple zeta function of depth $r$ can be extended as a
meromorphic function to $\C^r$ with possible simple poles at
the hyperplanes given by the equations
$$
s_1=1 ; ~s_1+\cdots+s_i =n \ \text{ for all }
\ n\in \Z_{\le i} \ \text{ and } \ 2 \le i \le r.
$$
Here $\Z_{\le i}$ denotes set of all integers less than or equal to $i$.
\end{thm}

In about the same time, S. Akiyama, S. Egami and Y. Tanigawa \cite{AET}
gave a simpler proof of the above fact using the classical Euler-Maclaurin
summation formula. They could also identify the exact set of
polar hyperplanes. The vanishing of the odd Bernoulli numbers
plays a central role in this context.

\begin{thm}[Akiyama-Egami-Tanigawa]\label{AET}
The multiple zeta function of depth $r$ is holomorphic in the open set
obtained by removing the following hyperplanes from $\C^r$
and it has simple poles at the hyperplanes given by the equations
\begin{equation*}
\begin{split}
& s_1=1 ; ~  s_1+s_2=2,1,0,-2,-4,-6, \ldots;\\
& s_1+\cdots+s_i = n \ \text{ for all } \ n \in \Z_{\le i} \ \text{ and } \ 3 \le i \le r.
\end{split}
\end{equation*}
\end{thm}

Thereafter, the problem of the meromorphic continuation 
of the multiple zeta functions drew attention of several
mathematicians. In this process,
a variety of methods evolved to address this problem. Perhaps the
simplest approach to this problem has been exhibited in \cite{MSV},
where the meromorphic continuation of the multiple zeta functions is
obtained by means of a simple translation formula and induction
on the depth $r$.

\begin{thm}[Mehta-Saha-Viswanadham] \label{anacont}
For each integer $r \ge 2$, the multiple zeta function of depth $r$
extends to a meromorphic function on $\mathbb C^r$ satisfying
the translation formula
\begin{equation} \label{transform}
\zeta_{r-1}(s_1+s_2-1,s_3,\ldots,s_r)=\sum_{k\ge 0}
(s_1-1)_k \ \zeta_r(s_1+k,s_2,\ldots,s_r),
\end{equation}
where the series of meromorphic functions on the right hand side
converges normally on all compact subsets of $\C^r$ and
for any $k \ge 0$ and $s\in \C$,  
$$
(s)_k:=\frac{s\cdots(s+k)}{(k+1)!}.
$$ 
\end{thm}

Further in \cite{MSV}, the method of
\textit{matrix formulation of translation formulas} was introduced
to write down the residues along the possible polar hyperplanes
in a computable form and recover the above
mentioned theorem of Akiyama, Egami and Tanigawa.

Soon after the works of Zhao \cite{JZ}, Akiyama, Egami
and Tanigawa \cite{AET}, several
generalisations of the multiple zeta functions were introduced
and their analytic properties were discussed. One such
important generalisation is the multiple Dirichlet series
associated to Dirichlet characters. It was first introduced
by Akiyama and Ishikawa \cite{AI} in 2002. Multiple Dirichlet series
associated to Dirichlet characters are now known as multiple
Dirichlet $L$-functions as these are several variable generalisations 
of the classical Dirichlet $L$-functions.

\begin{defn}
Let $r \ge 1$ be an integer and $\chi_1,\ldots,\chi_r$ be Dirichlet characters
of arbitrary modulus. The multiple Dirichlet $L$-function of depth $r$
associated to the Dirichlet characters $\chi_1,\ldots,\chi_r$ is denoted by 
$L_r(s_1,\ldots,s_r;~\chi_1,\ldots,\chi_r)$ 
and defined by the
following normally convergent series in $U_r$:
$$
L_r(s_1,\ldots,s_r;~\chi_1,\ldots,\chi_r):=\sum_{n_1>\cdots>n_r>0} 
\frac{\chi_1(n_1) \cdots \chi_r(n_r)}{n_1^{s_1} \cdots n_r^{s_r}}.
$$
\end{defn}

Akiyama and Ishikawa introduced the notion of
multiple Dirichlet $L$-functions for characters with same modulus,
but their definition also makes sense for Dirichlet characters
of arbitrary modulus. The normal convergence of the above series follows
from the normal convergence of the multiple zeta function of depth $r$
as an immediate consequence and we record it here in the following proposition.
Throughout this article, whenever we consider a set of characters,
they are of arbitrary modulus unless otherwise stated.

\begin{prop}\label{normal-conv-mdlf}
Let $r \ge 1$ be an integer and $\chi_1,\ldots,\chi_r$ be Dirichlet characters.
Then the family of functions
$$
\left ( \frac{\chi_1(n_1) \cdots \chi_r(n_r)}{n_1^{s_1} \cdots n_r^{s_r}}
\right )_{n_1>\cdots>n_r>0}
$$
is normally summable on compact subsets of $U_r$.
\end{prop}

So it follows that
$(s_1,\ldots,s_r) \mapsto L_r(s_1,\ldots,s_r;~\chi_1,\ldots,\chi_r)$
is a holomorphic function on $U_r$. Akiyama and Ishikawa have
also discussed the question of meromorphic continuation of the
multiple Dirichlet $L$-functions. When $r=1$,
classically the meromorphic continuation
is achieved by writing the function in terms of the Hurwitz zeta
function. For $r >1$, they pursued this idea under the assumption that
all the characters are of same modulus. In this
process some variants of the multiple Hurwitz zeta functions came up
and they proved the following theorem.

\begin{thm}[Akiyama-Ishikawa]\label{AI2}
Let $\chi_1,\ldots,\chi_r$ be Dirichlet characters of same modulus.
Then the multiple Dirichlet $L$-function $L_r(s_1,\ldots,s_r;~\chi_1,\ldots,\chi_r)$
of depth $r$ can be extended as a
meromorphic function to $\C^r$ with possible simple poles at
the hyperplanes given by the equations
$$
s_1=1 ; ~s_1+\cdots+s_i =n \ \text{ for all }
\ n\in \Z_{\le i} \ \text{ and } \ 2 \le i \le r.
$$
\end{thm}

At present, obtaining a complete description of the exact set of singularities
of the multiple Dirichlet $L$-functions seems to be a difficult problem.
For $r=2$ and specific choices of the characters $\chi_1$ and $\chi_2$,
Akiyama and Ishikawa were able to determine the exact set of polar hyperplanes.
To address this difficult question, one can aim to obtain a translation formula
satisfied by the multiple Dirichlet $L$-functions. For $r=1$, such a
translation formula has been established in \cite{BS}.

\begin{thm}[Saha]\label{periodic}
Let $\chi$ be a Dirichlet character of modulus $q$.
Then for $\Re(s)>1$, the associated Dirichlet $L$-function
$L(s;\chi):= \sum_{n \ge 1} \frac{\chi(n)}{n^s}$ satisfies the
following translation formula:
\begin{equation}\label{translation-periodic}
\sum_{a=1}^q \frac{\chi(a)}{a^{(s-1)}}=
\sum_{k \ge 0} (s-1)_k q^{k+1} \left( L(s+k;\chi) -
\sum_{a=1}^q \frac{\chi(a)}{a^{(s+k)}} \right),
\end{equation}
where the infinite series on the right hand side converges normally on every
compact subset of $\Re(s)>1$.
\end{thm}

As mentioned earlier, the meromorphic continuation of $L(s;\chi)$
is generally derived by writing $L(s;\chi)$ in
terms of the Hurwitz zeta function and then using the meromorphic
continuation of the Hurwitz zeta function. But one can also derive
this using \thmref{periodic}.
However, obtaining such a translation formula for multiple Dirichlet
$L$-functions seems harder and so is the analogue of \thmref{periodic}
for the multiple Dirichlet $L$-functions.

Now one can consider additive characters, i.e.
group homomorphisms $f: \Z \to \C^*$, in place of Dirichlet characters.
It is interesting to note that Dirichlet characters
are linked to additive characters and so is the
Dirichlet $L$-functions to the Dirichlet series associated 
to additive characters. Thus studying multiple Dirichlet series associated
to additive characters are very much relevant to the study of
multiple Dirichlet L-functions. In the following section
we introduce the notion of multiple Dirichlet series associated to
additive characters, which we call as multiple additive $L$-functions
and explore their relation with multiple Dirichlet $L$-functions.

In this article, we obtain the meromorphic continuation of the
multiple additive $L$-functions (see \thmref{anacont-malf}),
by means of translation formulas satisfied by these
multiple additive $L$-functions. Following the method of
\textit{matrix formulation of translation formulas} of \cite{MSV},
we first obtain an expression for residues along possible
polar hyperplanes of these multiple additive $L$-functions
(see \thmref{residues-malf}). We then determine the exact set
of singularities of these multiple additive $L$-functions
(see \thmref{exactpoles-malf}). For this we require
an important theorem of Frobenius about the zeros of the
Eulerian polynomials (see \S 6.3). Further, using the relation
between the multiple additive $L$-functions and the multiple
Dirichlet $L$-functions, we extend the above theorem of Akiyama
and Ishikawa for characters of arbitrary modulus (see \thmref{poles-mdlf}).

\section{Multiple additive $L$-functions}

\subsection{Relation with multiple Dirichlet $L$-functions}

Let $a,b,N$ be natural numbers with $1 \le a,b \le N$ and $\chi$ denote a
Dirichlet character of modulus $N$. For $\Re(s) >1$, we consider the 
following Dirichlet series:
\begin{align*}
& L(s;\chi) := \sum_{ n \ge 1} \frac{\chi(n)}{n^s},\\
& \Phi(s;a) := \sum_{\substack{ n \ge 1 \\ n \equiv a \bmod N}} \frac{1}{n^s},\\
& \Psi(s;b) := \sum_{ n \ge 1} \frac{e^{2 \pi \iota b n/N}}{n^s}.
\end{align*}
Then we have the following relations among these Dirichlet series:
\begin{equation}\label{L-Phi}
L(s;\chi)= \sum_{1 \le a \le N} \chi(a) \Phi(s;a),
\end{equation}
\begin{equation}\label{Psi-Phi}
\Psi(s;b)= \sum_{1 \le a \le N} e^{2 \pi \iota a b /N} \Phi(s;a).
\end{equation}
Further using \eqref{Psi-Phi}, one can deduce that
\begin{equation}\label{Phi-Psi}
\Phi(s;a)= \frac{1}{N} \sum_{1 \le b \le N} e^{-2 \pi \iota a b /N} \Psi(s;b).
\end{equation}
Hence
\begin{equation}\label{L-Psi}
L(s;\chi)= \frac{1}{N} \sum_{1 \le a \le N} \chi(a) \sum_{1 \le b \le N}
e^{-2 \pi \iota a b /N} \Psi(s;b).
\end{equation}

In fact \eqref{L-Psi} can be generalised for multiple Dirichlet $L$-functions.
Let $r \ge 1$ be a natural number and for each $1 \le i \le r$, let
$a_i,b_i,N_i$ be natural numbers with $1 \le a_i,b_i \le~N_i$.
Also for each $i$, let $\chi_i$ be a Dirichlet character of modulus $N_i$. Then the
depth-$r$ multiple Dirichlet $L$-function associated to $\chi_1,\ldots,\chi_r$
can be written as follows:
\begin{equation}\label{L-Phi-multiple}
L_r(s_1,\ldots,s_r;~\chi_1,\ldots,\chi_r)
= \sum_{\substack{1 \le a_i \le N_i \\ \text{for all } 1 \le i \le r}}
\chi_1(a_1) \cdots \chi_r(a_r) \Phi_r(s_1,\ldots,s_r; a_1,\ldots,a_r),
\end{equation}
where $\Phi_r(s_1,\ldots,s_r; a_1,\ldots,a_r)$ is defined as the 
following multiple Dirichlet series in $U_r$:
$$
\Phi_r(s_1,\ldots,s_r; a_1,\ldots,a_r) := \sum_{\substack{n_1>\cdots>n_r>0
\\ n_i \equiv a_i \bmod N_i \\ \text{for all } 1 \le i \le r}}
n_1^{-s_1} \cdots n_r^{-s_r}.
$$
Next we consider the following multiple Dirichlet series in $U_r$:
$$
\Psi_r(s_1,\ldots,s_r; b_1,\ldots,b_r) := \sum_{n_1>\cdots>n_r>0}
\frac{e^{2 \pi \iota \left( \frac{b_1 n_1}{N_1} + \cdots +\frac{b_r n_r}{N_r} \right)}}
{n_1^{s_1} \cdots n_r^{s_r}}.
$$
Now we can write down a several variable generalisation of \eqref{Psi-Phi}:
\begin{equation}\label{Psi-Phi-multiple}
\Psi_r(s_1,\ldots,s_r; b_1,\ldots,b_r)
=\sum_{\substack{1 \le a_i \le N_i \\ \text{for all } 1 \le i \le r}}
e^{2 \pi \iota \left( \frac{a_1 b_1}{N_1} + \cdots +\frac{a_r b_r}{N_r} \right)}
\Phi_r(s_1,\ldots,s_r; a_1,\ldots,a_r).
\end{equation}
Further using \eqref{Psi-Phi-multiple}, we derive
\begin{equation}\label{Phi-Psi-multiple}
\begin{split}
&\Phi_r(s_1,\ldots,s_r; a_1,\ldots,a_r)\\
&= \frac{1}{N_1 \cdots N_r}
\sum_{\substack{1 \le b_i \le N_i \\ \text{for all } 1 \le i \le r}}
e^{-2 \pi \iota \left( \frac{a_1 b_1}{N_1} + \cdots +\frac{a_r b_r}{N_r} \right)}
\Psi_r(s_1,\ldots,s_r; b_1,\ldots,b_r).
\end{split}
\end{equation}
Therefore, using \eqref{L-Phi-multiple} and \eqref{Phi-Psi-multiple} we obtain that
\begin{equation}\label{L-Psi-multiple}
\begin{split}
&L_r(s_1,\ldots,s_r;~\chi_1,\ldots,\chi_r)\\
&= \frac{1}{N_1 \cdots N_r}
\sum_{\substack{1 \le a_i \le N_i \\ \text{for all } 1 \le i \le r}}
\chi_1(a_1) \cdots \chi_r(a_r)
\sum_{\substack{1 \le b_i \le N_i \\ \text{for all } 1 \le i \le r}}
e^{-2 \pi \iota \left( \frac{a_1 b_1}{N_1} + \cdots +\frac{a_r b_r}{N_r} \right)}
\Psi_r(s_1,\ldots,s_r; b_1,\ldots,b_r).
\end{split}
\end{equation}
The above formula \eqref{L-Psi-multiple} plays an important role in the
investigation of the analytic properties of the multiple Dirichlet $L$-functions.

\subsection{Multiple additive $L$-functions and their translation formulas}

We now formally introduce the notion of multiple additive $L$-functions.
Besides the above exhibited relation \eqref{L-Psi-multiple} with
the multiple Dirichlet $L$-functions, the multiple additive $L$-functions
are also of independent interest. Special values of these functions,
such as the so-called cyclotomic multiple zeta values, the coloured
multiple zeta values have also been studied in arithmetic context.
These functions can also be thought of as cousins of the multiple
polylogarithms.

\begin{defn}
For a natural number $r \ge 1$ and additive characters
$f_1, \ldots, f_r$, the multiple additive $L$-function associated to $f_1, \ldots, f_r$
is denoted by $L_r(f_1, \ldots, f_r;~s_1,\ldots,s_r)$ and defined by the
following series:
$$
L_r(f_1, \ldots, f_r;~s_1,\ldots,s_r):=\sum_{n_1>\cdots>n_r>0} 
\frac{f_1(n_1) \cdots f_r(n_r)}{n_1^{s_1} \cdots n_r^{s_r}}
= \sum_{n_1>\cdots>n_r>0} 
\frac{f_1(1)^{n_1} \cdots f_r(1)^{n_r}}{n_1^{s_1} \cdots n_r^{s_r}}.
$$
\end{defn}

A necessary and sufficient condition for the absolute convergence of the
above series in $U_r$ is given in terms of the partial products
$g_i:= \prod_{1 \le j \le i} f_j$ for all $1 \le i \le r$.
The condition is that
$$
|g_i(1)| \le 1 ~ \text{for all} ~ 1 \le i \le r.
$$
The sufficiency of this condition is easily established by noting that
for any arbitrary $r$ integers $n_1>\cdots>n_r>0$, one has
$$
\left| f_1(1)^{n_1} \cdots f_r(1)^{n_r} \right|
\le \left| g_1(1)^{n_1-n_2} \cdots g_{r-1}(1)^{n_{r-1}-n_r} g_r(1)^{n_r} \right|.
$$
It can also be shown that if $|g_i(1)| >1$
for some $i$, the above series does not converge absolutely in $U_r$,
in fact, anywhere in $\C^r$. To see this, let $i$ be one of the index such that
$|g_i(1)| >1$. Now if possible let for a complex $r$-tuple $(s_1, \ldots, s_r)$,
the series
$$
\sum_{n_1>\cdots>n_r>0} \frac{f_1(n_1) \cdots f_r(n_r)}
{n_1^{s_1} \cdots n_r^{s_r}}
$$
converges absolutely i.e. the series
$$
\sum_{n_1>\cdots>n_r>0} \frac{|f_1(n_1) \cdots f_r(n_r)|}
{n_1^{\sigma_1} \cdots n_r^{\sigma_r}}
$$
of non-negative real numbers converges, where $\sigma_i$ denotes the
real part of $s_i$. Then the following smaller series
$$
\sum_{n > r-i} \frac{|f_1(n+i-1) f_2(n+i-2) \cdots f_i(n)
f_{i+1}(r-i) \cdots f_r(1)|}
{(n+i-1)^{\sigma_1} (n+i-2)^{\sigma_2} \cdots
n^{\sigma_i} (r-i)^{\sigma_{i+1}} \cdots
1^{\sigma_r}}
$$
is convergent.
Note that the numerator of the summand in the above series
is nothing but
$$
|g_1(1) \cdots g_{i-1}(1) g_i(1)^n f_{i+1}(r-i) \cdots f_r(1)|
$$
and the denominator is smaller than
$$
(n+i-1)^{\sigma_1 + \cdots + \sigma_i} (r-i)^{\sigma_{i+1}} \cdots 1^{\sigma_r}.
$$
As the series
$$
\sum_{n > r-i} \frac{|g_i(1)|^n}{(n+i-1)^{\sigma_1 + \cdots + \sigma_i}}
$$
does not converge for any choice of complex $r$-tuple $(s_1, \ldots, s_r)$
when $|g_i(1)|>1$, we get a contradiction.

If we write $f_i(1)= e^{2 \pi \iota \lambda_i}$ for some $\lambda_i \in \C$,
the condition $|g_i(1)| \le 1$ for $1 \le i \le r$ can be rewritten as
$$
\Im(\lambda_1 + \cdots + \lambda_i) \ge 0 \ \text{ for } 1 \le i \le r.
$$
With the necessary and sufficient condition for absolute convergence
of the multiple additive $L$-functions in place, we derive the following result
as an immediate consequence of the normal convergence of the multiple
zeta function of depth $r$.

\begin{prop}\label{normal-conv-malf}
Let $r \ge 1$ be an integer and $f_1,\ldots,f_r$ be additive characters
such that the partial products $g_i:= \prod_{1 \le j \le i} f_j$ satisfy
the condition $|g_i(1)| \le 1$ for all  $1 \le i \le r$.
Then the family of functions
$$
\left ( \frac{f_1(n_1) \cdots f_r(n_r)}{n_1^{s_1} \cdots n_r^{s_r}}
\right )_{n_1>\cdots>n_r>0}
$$
is normally summable on compact subsets of $U_r$.
\end{prop}

For sake of completeness, we recall the definition of normal
convergence.

\begin{defn}
Let $X$ be a set and $(f_i)_{i \in I}$ be a 
family of complex valued functions
defined on $X$. We say that the family of 
functions $(f_i)_{i \in I}$
is normally summable on $X$ or the series $\sum_{i \in I} f_i$ 
converges normally on $X$ if 
$$
\|f_i\|_X := \sup_{x \in X} |f(x)| < \infty ,
~\text{      for all  }i \in I
$$ 
and the family of real numbers 
$(\| f_i \|_X)_{i \in I}$ is summable. 
\end{defn}

\begin{defn}
Let $X$ be an open subset of $\C^r$ and $(f_i)_{i \in I}$ be a
family of meromorphic functions on $X$. We say that
$(f_i)_{i \in I}$ is normally summable or $\sum_{ i \in I} f_i$
is normally convergent on all compact
subsets of $X$ if for any compact subset $K$ of $X$,
there exists a finite set $J \subset I$ such that
each $f_i$ for $i \in I \setminus J$ is holomorphic in an open 
neighbourhood of $K$ and the family 
$(f_i|K)_{ i \in {I \setminus J}}$ is normally summable
on $K$. In this case, $\sum_{ i \in I} f_i$ is a well
defined meromorphic function on~$X$.
\end{defn}

\propref{normal-conv-malf} implies that for a
natural number $r \ge 1$ and additive characters
$f_1, \ldots, f_r$ such that the partial products
$g_i:= \prod_{1 \le j \le i} f_j$ satisfy the condition
$|g_i(1)| \le 1$ for all  $1 \le i \le r$, the multiple additive $L$-function
$L_r(f_1, \ldots, f_r;~s_1,\ldots,s_r)$ defines an analytic function in $U_r$.
We now address the question of meromorphic continuation of such multiple
additive $L$-functions. This is done via obtaining translation formulas
analogous to \eqref{transform}. The multiple additive $L$-function
$L_r(f_1, \ldots, f_r;~s_1,\ldots,s_r)$
satisfies the following translation formulas depending
on the condition that whether $f_1(1)=1$ or not.

\begin{thm}\label{translation-malf-1}
For any integer $r \ge 2$ and additive characters $f_1, \ldots, f_r$ such that
the partial products $g_i:= \prod_{1 \le j \le i} f_j$ satisfy the condition
$|g_i(1)| \le 1$ for all  $1 \le i \le r$ with $f_1(1)=1$, the associated
multiple additive $L$-function $L_r(f_1, \ldots, f_r;~s_1,\ldots,s_r)$
satisfies the following translation formula in $U_r$:
\begin{equation}\label{tf-malf-1}
L_{r-1}(f_2, \ldots, f_r;~s_1+s_2-1,s_3,\ldots,s_r)
=\sum_{k\ge 0} (s_1-1)_k \
L_r(f_1, \ldots, f_r;~s_1+k,s_2,\ldots,s_r),
\end{equation} 
where the series on the right hand side converges normally
on any compact subset of $U_r$ .
\end{thm}

Note that the translation formula \eqref{tf-malf-1} is nothing but the
translation formula \eqref{transform}, written for
$L_r(f_1, \ldots, f_r;~s_1,\ldots,s_r)$ in place of $\z_r(s_1,\ldots,s_r)$.
If $f_1(1) \neq 1$, then we have the following theorem.

\begin{thm}\label{translation-malf-2}
Let $r \ge 2$ be an integer and $f_1, \ldots, f_r$ be  additive characters such that
the partial products $g_i:= \prod_{1 \le j \le i} f_j$ satisfy the condition
$|g_i(1)| \le 1$ for all  $1 \le i \le r$ with $f_1(1) \neq 1$. Then the associated
multiple additive $L$-function $L_r(f_1, \ldots, f_r;~s_1,\ldots,s_r)$
satisfies the following translation formula in $U_r$:
\begin{equation}\label{tf-malf-2}
\begin{split}
& f_1(1) L_{r-1}(g_2,f_3, \ldots, f_r;s_1+s_2,s_3,\ldots,s_r)
+(f_1(1)-1) L_r(f_1, \ldots, f_r;s_1,\ldots,s_r) \\
&=  \sum_{k\ge 0} (s_1)_k \ L_r(f_1, \ldots, f_r;~s_1+k+1,s_2,\ldots,s_r),
\end{split}
\end{equation} 
where the series on the right side converges normally
on any compact subset of $U_r$ .
\end{thm}

\section{Proof of the translation formulas}
To prove \thmref{translation-malf-1}, we need the
identity which is valid for any integer $n \ge 2$ and any complex
number $s$:
\begin{equation}\label{trick-malf-1}
(n-1)^{1-s}-n^{1-s}= \sum _{k\ge 0} (s-1)_k \ n^{-s-k}.
\end{equation}
Whereas to prove \thmref{translation-malf-2}, we need the following version
of \eqref{trick-malf-1}, obtained by replacing $s$ with $s+1$
in \eqref{trick-malf-1}:
\begin{equation}\label{trick-malf-2}
(n-1)^{-s}-n^{-s}= \sum _{k\ge 0} (s)_k \ n^{-s-k-1}.
\end{equation}
The identity \eqref{trick-malf-1} is obtained by writing the left hand side as
$n^{1-s}\left( (1-\frac{1}{n})^{1-s}-1 \right)$ and expanding
$(1-\frac{1}{n})^{1-s}$ as a Taylor series in $\frac{1}{n}$.
Besides, we need the following proposition.

\begin{prop}\label{normal-conv-malf-1}
Let $r \ge 2$ be an integer and $f_1, \ldots, f_r$ be  additive characters such that
the partial products $g_i:= \prod_{1 \le j \le i} f_j$ satisfy the condition
$|g_i(1)| \le 1$ for all  $1 \le i \le r$. Then the family of functions
$$
\left ( (s_1-1)_k \frac{f_1(n_1) \cdots f_r(n_r)}{n_1^{s_1+k}
n_2^{s_2} \cdots n_r^{s_r}} \right )_{n_1>\cdots>n_r>0, k\ge 0}
$$
is normally summable on compact subsets of $U_r$.
\end{prop}

\begin{proof}
From the hypothesis we have
$$
\left| f_1(n_1) \cdots f_r(n_r) \right| \le 1.
$$
Now let $K$ be a compact subset of $U_r$ and set
$$
S :=  \sup_{ (s_1,\ldots,s_r) \in K} |s_1 - 1 |. 
$$ 
Since $r \ge 2$, for any strictly decreasing sequence 
$n_1, \ldots,n_r$ of $r$ positive integers, we have $n_1 \ge 2$. 
Hence for $k \ge 0$, we get
$$
\left\| (s_1-1)_k \ n_1^{-s_1-k}n_2^{-s_2}\cdots n_r^{-s_r} \right\|_K
~\le~
\frac{(S)_k}{2^k} \|n_1^{-s_1}\cdots n_r^{-s_r}\|_K.
$$
Thus the proof follows from the normal convergence
of the multiple zeta function of depth $r$ and the fact that
for any real number $a$, the series
$$
\sum_{k\ge 0} \frac{(a)_k}{2^k}
$$
is convergent.
\end{proof}

We are now ready to prove \thmref{translation-malf-1}
and \thmref{translation-malf-2}.

\subsection{Proof of \thmref{translation-malf-1}}
We replace $n,s$ by $n_1,s_1$ in \eqref{trick-malf-1}  and
multiply $\frac{f_1(n_1) \cdots f_r(n_r)}{n_2^{s_2} \cdots n_r^{s_r}}$
to both sides of \eqref{trick-malf-1} and obtain that
$$
\left( \frac{1}{(n_1-1)^{s_1-1}} - \frac{1}{n_1^{s_1-1}} \right) 
\frac{f_1(n_1) \cdots f_r(n_r)}{n_2^{s_2} \cdots n_r^{s_r}}
=\sum_{k \ge 0} (s_1-1)_k \frac{f_1(n_1) \cdots f_r(n_r)}{n_1^{s_1+k}
n_2^{s_2} \cdots n_r^{s_r}}.
$$
Now we sum for $n_1>\cdots>n_r>0$. Since $f_1(1)=1$,
using \propref{normal-conv-malf-1}, we get
$$
L_{r-1}(f_2, \ldots, f_r;~s_1+s_2-1,s_3,\ldots,s_r)
=\sum_{k\ge 0} (s_1-1)_k \
L_r(f_1, \ldots, f_r;~s_1+k,s_2,\ldots,s_r).
$$
This together with \propref{normal-conv-malf-1} completes the proof. \qed

\subsection{Proof of \thmref{translation-malf-2}}
For this we replace $n,s$ by $n_1,s_1$ in \eqref{trick-malf-2}  and
multiply $\frac{f_1(n_1) \cdots f_r(n_r)}{n_2^{s_2} \cdots n_r^{s_r}}$
to both sides of \eqref{trick-malf-2} and obtain that
$$
\left( \frac{1}{(n_1-1)^{s_1}} - \frac{1}{n_1^{s_1}} \right) 
\frac{f_1(n_1) \cdots f_r(n_r)}{n_2^{s_2} \cdots n_r^{s_r}}
=\sum_{k \ge 0} (s_1)_k \frac{f_1(n_1) \cdots f_r(n_r)}{n_1^{s_1+k+1}
n_2^{s_2} \cdots n_r^{s_r}}.
$$
As before, we sum for $n_1>\cdots>n_r>0$ and use
\propref{normal-conv-malf-1} with $s_1$ replaced by $s_1+1$.
Since $f_1(1) \neq 1$, we obtain
\begin{equation*}
\begin{split}
& f_1(1) L_{r-1}(g_2,f_3, \ldots, f_r;s_1+s_2,s_3,\ldots,s_r)
+(f_1(1)-1) L_r(f_1, \ldots, f_r;s_1,\ldots,s_r) \\
&=  \sum_{k\ge 0} (s_1)_k \ L_r(f_1, \ldots, f_r;~s_1+k+1,s_2,\ldots,s_r).
\end{split}
\end{equation*}
This together with \propref{normal-conv-malf-1} completes the proof. \qed

\section{Meromorphic continuation}
In this section, we establish the meromorphic continuation of
the multiple additive $L$-functions using the translation
formulas \eqref{tf-malf-1} and \eqref{tf-malf-2}.

\begin{thm}\label{anacont-malf}
Let $r \ge 2$ be an integer and $f_1, \ldots, f_r$ be  additive characters such that
the partial products $g_i:= \prod_{1 \le j \le i} f_j$ satisfy the condition
$|g_i(1)| \le 1$ for all  $1 \le i \le r$. Then the associated
multiple additive  $L$-function $L_r(f_1, \ldots, f_r;~s_1,\ldots,s_r)$ extends
to a meromorphic function on $\C^r$ satisfying the following translation formulas
on $\C^r$:
\begin{equation}\label{tf-malf-3}
\begin{split}
& L_{r-1}(f_2, \ldots, f_r;~s_1+s_2-1,s_3,\ldots,s_r) \\
& =\sum_{k\ge 0} (s_1-1)_k \
L_r(f_1, \ldots, f_r;~s_1+k,s_2,\ldots,s_r)
\phantom{m} \text{if} \phantom{m} f_1(1)=1 
\end{split}
\end{equation}
and
\begin{equation}\label{tf-malf-4}
\begin{split}
& f_1(1) L_{r-1}(g_2,f_3, \ldots, f_r;s_1+s_2,s_3,\ldots,s_r)
+(f_1(1)-1) L_r(f_1, \ldots, f_r;s_1,\ldots,s_r) \\
& =  \sum_{k\ge 0} (s_1)_k \ L_r(f_1, \ldots, f_r;~s_1+k+1,s_2,\ldots,s_r)
\phantom{m} \text{if} \phantom{m} f_1(1) \neq 1.
\end{split}
\end{equation}
The series of meromorphic functions  on the right hand sides of
\eqref{tf-malf-3} and \eqref{tf-malf-4}  converge normally on every compact
subset of $\C^r$.
\end{thm}

\begin{rmk}\label{anacont-alf}
\rm  When $r=1$ and $f_1(1)=1$, the associated additive $L$-function
is nothing but the Riemann zeta function, for which we know that
it can be meromorphically continued to $\C$ with only
a simple pole at $s=1$ with residue $1$.
When $r = 1$ and $f_1(1) \neq 1$, it is also known
that the associated additive $L$-function has an analytic continuation
to the entire complex plane (for a proof see \cite{BS}).
\end{rmk}

We prove \thmref{anacont-malf} by induction on depth $r$.
Assuming induction hypothesis for multiple additive  $L$-functions
of depth $(r-1)$, we first extend the multiple additive  $L$-function
$L_r(f_1, \ldots, f_r;~s_1,\ldots,s_r)$ as a meromorphic function
to $U_r(m)$ for each $m \ge 0$, where
$$
U_r(m):=\{(s_1,\ldots,s_r) \in \C^r ~:~ \Re(s_1+\cdots +s_i) > i - m ~
\text{ for all } ~ 1 \le i \le r\}.
$$
Since  open sets of the form $U_r(m)$ form an open cover of $\C^r$,
we get the coveted meromorphic continuation to $\C^r$.
In this case we need the following variant of  \propref{normal-conv-malf-1}.

\begin{prop}\label{normal-conv-malf-2}
Let $r \ge 2$ be an integer and $f_1, \ldots, f_r$ be  additive characters such that
the partial products $g_i:= \prod_{1 \le j \le i} f_j$ satisfy the condition
$|g_i(1)| \le 1$ for all  $1 \le i \le r$. Then the family of functions
$$
\left ( (s_1)_k \frac{f_1(n_1) \cdots f_r(n_r)}{n_1^{s_1+k+1}
n_2^{s_2} \cdots n_r^{s_r}} \right )_{n_1>\cdots>n_r>0, k\ge m-1}
$$
is normally summable on compact subsets of $U_r(m)$.
\end{prop}

\begin{proof}
Let $K$ be a compact subset of $U_r(m)$ and
$S:=\sup_{(s_1,\ldots,s_r) \in K} |s_1|$. Then for $k \ge m-1$,
$$
\left\|  (s_1)_k \frac{f_1(n_1) \cdots f_r(n_r)}{n_1^{s_1+k+1}
n_2^{s_2} \cdots n_r^{s_r}} \right \|
\le \frac{(S)_k}{2^{k-m+1}} \left\|  \frac{1}{n_1^{s_1+m}
n_2^{s_2} \cdots n_r^{s_r}} \right \|.
$$
Now as $(s_1,\ldots,s_r)$ varies over $U_r(m)$,
$(s_1+m,\ldots,s_r)$ varies over $U_r$. Then the proof follows
from the normal convergence
of the multiple zeta function of depth $r$ as the series
$\sum_{k \ge m-1}  \frac{(S)_k}{2^{k-m+1}}$ converges.
\end{proof}

\subsection{Proof of \thmref{anacont-malf}}
As mentioned before, we prove this theorem by induction on depth $r$.
For $r=2$, the left hand sides of \eqref{tf-malf-3} and \eqref{tf-malf-4}
have a meromorphic continuation to $\C^2$ by \rmkref{anacont-alf}.
If $r \ge 3$, then the left hand side of \eqref{tf-malf-3} and
the first term in the left hand side of \eqref{tf-malf-4}
have a meromorphic continuation to $\C^r$ by the induction hypothesis.

We now establish the meromorphic
continuation of the multiple additive $L$-function \linebreak
$L_r(f_1, \ldots, f_r;~s_1,\ldots,s_r)$ separately for each of the cases
$f_1(1)=1$ and $f_1(1) \neq 1$.

First we consider the case $f_1(1)=1$. As we have shown that
 in this case the
multiple additive  $L$-function $L_r(f_1, \ldots, f_r;~s_1,\ldots,s_r)$
satisfies the translation formula \eqref{tf-malf-1} in $U_r$.
Note that the translation formula \eqref{tf-malf-1} is exactly the
same as \eqref{transform} which is satisfied by the multiple 
zeta functions of depth $r$.
Hence in this case the meromorphic continuation follows
exactly as in the case of the multiple zeta function
and therefore we omit the proof.

Next we consider the case $f_1(1) \not =1$. Now from
induction hypothesis we know that
$$
f_1(1) L_{r-1}(g_2,f_3, \ldots, f_r;s_1+s_2,s_3,\ldots,s_r)
$$
has a meromorphic continuation to $\C^r$ and by
\propref{normal-conv-malf-2} (for $m=1$), we have that
 $$
 \sum_{k\ge 0} (s_1)_k \ L_r(f_1, \ldots, f_r;~s_1+k+1,s_2,\ldots,s_r)
$$
is a holomorphic function on $U_r(1)$. Hence we can extend
$L_r(f_1, \ldots, f_r;~s_1,\ldots,s_r)$ to $U_r(1)$ as a meromorphic
function satisfying \eqref{tf-malf-2}.

Now we assume that  for an integer
$m \ge 2$, $L_r(f_1, \ldots, f_r;~s_1,\ldots,s_r)$
has been extended meromorphically to $U_r(m-1)$ satisfying
\eqref{tf-malf-2}. Again by \propref{normal-conv-malf-2}, the series
$$
\sum_{k\ge m-1} (s_1)_k \ L_r(f_1, \ldots, f_r;~s_1+k+1,s_2,\ldots,s_r)
$$
defines a holomorphic function on $U_r(m)$.
As $L_r(f_1, \ldots, f_r;~s_1,\ldots,s_r)$
has been extended meromorphically to $U_r(m-1)$,
we can extend $L_r(f_1, \ldots, f_r;~s_1+k+1,\ldots,s_r)$
meromorphically to $U_r(m)$ for all $0 \le k \le m-2$.

Thus we can now extend $L_r(f_1, \ldots, f_r;~s_1,\ldots,s_r)$
meromorphically to $U_r(m)$ by means of \eqref{tf-malf-2}.
This completes the proof as  $\{U_r(m): m \ge 1\}$ is an
open cover of $\C^r$. \qed

\section{Matrix formulation of the translation formulas}
The method of \textit{matrix formulation of translation formulas}
was introduced in \cite{MSV} in order to write down the translation
formula \eqref{transform} in an expression involving infinite matrices.
We adapt this method to obtain similar matrix formulation of
the translation formulas \eqref{tf-malf-3} and \eqref{tf-malf-4}.
More details about the method of matrix formulation of translation formulas
can be found in \cite{MSV}.

\subsection{Matrix formulation of \eqref{tf-malf-3}}

We recall that the translation formula \eqref{tf-malf-3} is nothing but the
translation formula \eqref{transform}, written for
$L_r(f_1, \ldots, f_r;~s_1,\ldots,s_r)$ in place of $\z_r(s_1,\ldots,s_r)$.
Thus the matrix formulation for \eqref{tf-malf-3} will be
identical to the one for \eqref{transform}. For sake of completeness
we discuss this in brief.

The translation formula \eqref{tf-malf-3} together with
the other relations obtained by applying successively the
change of variables $s_1 \mapsto s_1+n$ to it  for
each  $n\geq 0$ is equivalent to the single relation
\begin{equation} \label{mat-tf-malf-1}
{\bf V}_{r-1}(f_2,f_3,\ldots,f_r; \ s_1+s_2-1,s_3,\ldots,s_r)
={\bf A_1}(s_1-1) {\bf V}_r(f_1,\ldots,f_r;\ s_1,\ldots,s_r),
\end{equation}
where ${\bf V}_r(f_1,\ldots,f_r;\ s_1,\ldots,s_r)$ denotes the 
infinite column vector
\begin{equation} \label{eqV-malf}
{\bf V}_r(f_1,\ldots,f_r;\ s_1,\ldots,s_r) := \left( \begin{array}{c}
L_r(f_1,\ldots,f_r;\ s_1,s_2,\ldots,s_r)\\
L_r(f_1,\ldots,f_r;\ s_1+1,s_2,\ldots,s_r)\\
L_r(f_1,\ldots,f_r;\ s_1+2,s_2,\ldots,s_r)\\
\vdots
\end{array} \right)
\end{equation}
and for an indeterminate $t$, ${\bf A_1}(t)$ is defined by
\begin{equation}\label{eqA1-malf}
{\bf A_1}(t) := \left( \begin{array}{c c c c}
t & \frac{t(t+1)}{2!} & \frac{t(t+1)(t+2)}{3!} & \cdots\\
0 & t+1 & \frac{(t+1)(t+2)}{2!} & \cdots\\
0 & 0 & t+2 & \cdots\\
\vdots & \vdots & \vdots & \ddots
\end{array} \right).
\end{equation}
The matrix ${\bf A_1}(t)$ is invertible in $\mathbf{T}(\Q(t))$ and
its inverse matrix ${\bf B_1}(t)$ is given by
\begin{equation} \label{eqBt}
{\bf B_1}(t) = \left( \begin{array}{c c c c c}
\frac{1}{t} & \frac{B_1}{1!} &  \frac{(t+1)B_2}{2!} &
\frac{(t+1)(t+2)B_3}{3!} & \cdots \\
0 & \frac{1}{t+1} & \frac{B_1}{1!} & \frac{(t+2)B_2}{2!} & \cdots \\
0 & 0 & \frac{1}{t+2} & \frac{B_1}{1!} & \cdots \\
0 & 0 & 0 & \frac{1}{t+3} & \cdots\\
\vdots & \vdots & \vdots & \vdots & \ddots
\end{array} \right).
\end{equation}
Here $\mathbf{T}(\Q(t))$ denotes the ring of upper triangular matrices
of type $\mathbb N \times \mathbb N$ with coefficients in $\Q(t)$.

This suggests that one may attempt to express the column vector
${\bf V}_r(f_1,\ldots,f_r;\ s_1,\ldots,s_r)$ in terms of
${\bf V}_{r-1}(f_2,f_3,\ldots,f_r; \ s_1+s_2-1,s_3,\ldots,s_r)$
by multiplying both sides of \eqref{mat-tf-malf-1} by ${\bf B_1}(s_1-1)$.
However this is not allowed. The reason is that the entries of the
formal product of the matrix ${\bf B_1}(s_1-1)$ with the column vector
${\bf V}_{r-1}(s_1+s_2-1,s_3,\ldots,s_r)$ are not convergent series.

To get around this difficulty we choose an integer $q \ge 1$ and define
$$
I = I_q := \{k \in \N : 0 \le k \le q-1\}  \ \text{ and } \ 
J = J_q := \{k \in \N : k \ge q\}.
$$
This allows us to write the previous matrices
as block matrices, for example
$$
{\bf A_1}(t) = \left( \begin{array}{c c}
{\bf A_1}^{II}(t) & {\bf A_1}^{IJ}(t)\\
{\bf 0}^{JI} & {\bf A_1}^{JJ}(t)
\end{array} \right).
$$
With these notations, we rewrite \eqref{mat-tf-malf-1} as
follows:
\begin{equation*}
\begin{split}
&{\bf V}_{r-1}^I (f_2,f_3,\ldots,f_r; \ s_1+s_2-1,s_3,\ldots,s_r)\\
&={\bf A_1}^{II}(s_1-1) {\bf V}_r^I(f_1,\ldots,f_r;\ s_1,\ldots,s_r)
+{\bf A_1}^{IJ}(s_1-1) {\bf V}_r^J(f_1,\ldots,f_r;\ s_1,\ldots,s_r).
\end{split}
\end{equation*}
Now inverting ${\bf A_1}^{II}(s_1-1)$, we can
reformulate \eqref{mat-tf-malf-1} to write
\begin{equation}\label{vbvw-malf}
\begin{split}
&{\bf V}_r^I(f_1,\ldots,f_r;\ s_1,\ldots,s_r) \\
&= {\bf B_1}^{II}(s_1-1)
{\bf V}_{r-1}^I(f_2,\ldots,f_r;\ s_1+s_2-1,s_3,\ldots,s_r)
+ {\bf W}^I (f_1,\ldots,f_r;\ s_1,\ldots,s_r),
\end{split}
\end{equation}
where
\begin{equation}\label{wuuv-malf}
{\bf W}^I (f_1,\ldots,f_r;\ s_1,\ldots,s_r) =  - {\bf B_1}^{II}(s_1-1) {\bf A_1}^{IJ}(s_1)
{\bf V}_r^J(f_1,\ldots,f_r;\ s_1,\ldots,s_r).
\end{equation}
All the series of meromorphic functions involved in the products of
matrices in formulas \eqref{vbvw-malf} and \eqref{wuuv-malf} converge normally
on all compact subsets of $\C^r$. Moreover, all the entries of the
matrices on the right hand side of \eqref{wuuv-malf} are holomorphic in
the open set $U_r(q)$, the translate of $U_r$ by $(-q,0,\ldots,0)$. Therefore
the entries of ${\bf W}^I(f_1,\ldots,f_r;\ s_1,\ldots,s_r)$ are also holomorphic
in $U_r(q)$.

If we write $\xi_q(f_1,\ldots,f_r;\ s_1,\ldots,s_r)$ to be the first entry
of ${\bf W}^I(f_1,\ldots,f_r;\ s_1,\ldots,s_r)$,
we then get from \eqref{vbvw-malf} that
\begin{equation}\label{explicit-malf-1}
\begin{split}
& L_r(f_1,\ldots,f_r;\ s_1,\ldots,s_r) \\
& =\frac{1}{s_1-1} \ L_{r-1}(f_2,\ldots,f_r;\ s_1+s_2-1,s_3,\ldots,s_r)\\
& \ \ + \sum_{k=0}^{q-2} \frac{s_1\cdots (s_1+k-1)}{(k+1)!} B_{k+1} \
L_{r-1}(f_2,\ldots,f_r;\ s_1+s_2+k,s_3,\ldots,s_r)\\
& \ \ + \xi_q(f_1,\ldots,f_r;\ s_1,\ldots,s_r),
\end{split}
\end{equation}
where $\xi_q(f_1,\ldots,f_r;\ s_1,\ldots,s_r)$ is holomorphic 
in the open set $U_r(q)$.

\subsection{Matrix formulation of \eqref{tf-malf-4}}

On the other hand, the translation formula \eqref{tf-malf-4}
and the other relations obtained by applying successively the
change of variable $s_1 \mapsto s_1+n$ to it for
each  $n\geq 0$, is equivalent to the relation
\begin{equation} \label{mat-tf-malf-2}
f_1(1) {\bf V}_{r-1}(g_2,f_3,\ldots,f_r; \ s_1+s_2,s_3,\ldots,s_r)
={\bf A_2}(f_1;s_1) {\bf V}_r(f_1,\ldots,f_r;\ s_1,\ldots,s_r),
\end{equation}
where for an indeterminate $t$ and a non-trivial group homomorphism
$f:\Z \to \C^*$, the matrix ${\bf A_2}(f;t)$ 
is defined as follows:
\begin{equation}
\label{eqA3-malf}
{\bf A_2}(f;t) := \left( \begin{array}{c c c c c}
1-f(1) & t & \frac{t(t+1)}{2!} & \frac{t(t+1)(t+2)}{3!} & \cdots\\
0 & 1-f(1) & t+1 & \frac{(t+1)(t+2)}{2!} & \cdots\\
0 & 0 & 1-f(1) & t+2 & \cdots\\
0 & 0 & 0 & 1-f(1) & \cdots\\
\vdots & \vdots & \vdots & \vdots & \ddots
\end{array} \right).
\end{equation}
As $f(1) \neq 1$, the matrix ${\bf A_2}(f;t)$ is invertible in ${\bf T}(\C(t))$,
the ring of upper triangular matrices of type $\N \times \N$
with coefficients in $\C(t)$. To find
the inverse of ${\bf A_2}(f;t)$ explicitly, we notice that
$$
{\bf A_2}(f;t)=e({\bf M}(t)) - f(1) {\bf I}_{\N \times \N},
$$
where ${\bf I}_{\N \times \N}$
denotes the identity matrix in ${\bf T}(\C(t))$
and
$$
{\bf M}(t) := \left( \begin{array}{c c c c}
0 & t & 0 & \cdots\\
0 & 0 & t+1 & \cdots\\
0 & 0 & 0 & \cdots\\
\vdots & \vdots & \vdots & \ddots
\end{array} \right).
$$
Now we need to invert the power series $e^x - c$
in $\C[[x]]$ for some $c \neq 1$. For the sake of convenience,
let us replace the variable $x$ by $(c-1)y$, where $c \neq 1$.
Our aim is to invert the power series $e^{(c-1)y}-c$ when $c \neq 1$. 
For this we recall the generating function for the Eulerian polynomials. The
Eulerian polynomials $A_n(t)$'s are defined by the following exponential
generating function
$$
\sum_{n \ge 0} A_n(t) \frac{y^n}{n!}
= \frac{1-t}{e^{(t-1)y}-t}.
$$
Thus
$$
\frac{1}{e^{(c-1)y}-c}
= \frac{1}{1-c} \sum_{n \ge 0} A_n(c) \frac{y^n}{n!}
\phantom{m}\text{as}\phantom{m}
c \neq 1.
$$
Hence
$$
\frac{1}{e^{x}-c}
=\frac{1}{1-c} \sum_{n \ge 0} A_n(c) \frac{x^n}{(c-1)^n n!}
\phantom{m}\text{when}\phantom{m}
c \neq 1.
$$
This in turn yields that the inverse of ${\bf A_2}(f;t)$,
which we denote by ${\bf B_2}(f;t)$, is given by the formula
\begin{equation}
\label{eqB3-malf}
{\bf B_2}(f;t) = \frac{1}{1-f(1)}  \sum_{n \ge 0}
A_n(f(1)) \frac{{\bf M}(t)^n}{(f(1)-1)^n n!}
\end{equation}
as $f(1) \neq 1$. It can be calculated that $A_0(t)=A_1(t)=1$. Hence
$$
{\bf B_2}(f;t) = \frac{1}{1-f(1)}
 \left( \begin{array}{c c c c c}
1 & \frac{1}{f(1)-1}t
& \frac{A_2(f(1))}{(f(1)-1)^2} \frac{t(t+1)}{2!}
& \frac{A_3(f(1))}{(f(1)-1)^3}\frac{t(t+1)(t+2)}{3!}
& \cdots \\
0 & 1 & \frac{1}{f(1)-1} (t+1)
& \frac{A_2(f(1))}{(f(1)-1)^2} \frac{(t+1)(t+2)}{2!} & \cdots \\
0 & 0 & 1 & \frac{1}{f(1)-1} (t+2) & \cdots \\
0 & 0 & 0 & 1 & \cdots\\
\vdots & \vdots & \vdots & \vdots & \ddots
\end{array} \right).
$$
Here also one cannot multiply both sides of \eqref{mat-tf-malf-2}
by ${\bf B_2}(f_1;s_1)$ to express the column vector
${\bf V}_r(f_1,\ldots,f_r;\ s_1,\ldots,s_r)$ in terms of
${\bf V}_{r-1}(g_2,f_3,\ldots,f_r;\ s_1+s_2,s_3,\ldots,s_r)$.
This is not possible as the coefficients of the Eulerian
polynomials grow very fast. In fact it is known
that for each $n \ge 0$, the sum of the coefficients of $A_n(t)$ is $n!$.

To get around this difficulty we repeat our truncation process.
For an integer $q \ge 1$, let $I = I_q,J=J_q$ be as before.
Then we rewrite \eqref{mat-tf-malf-2} as
follows:
\begin{equation*}
\begin{split}
&f_1(1) {\bf V}_{r-1}^I (g_2,f_3,\ldots,f_r; \ s_1+s_2,s_3,\ldots,s_r)\\
&={\bf A_2}^{II}(f_1;s_1) {\bf V}_r^I(f_1,\ldots,f_r;\ s_1,\ldots,s_r)
+{\bf A_2}^{IJ}(f_1;s_1) {\bf V}_r^J(f_1,\ldots,f_r;\ s_1,\ldots,s_r).
\end{split}
\end{equation*}
Now inverting ${\bf A_2}^{II}(f_1;s_1)$, we get
\begin{equation}\label{vbvz}
\begin{split}
&{\bf V}_r^I(f_1,\ldots,f_r;\ s_1,\ldots,s_r)\\
&=f_1(1) {\bf B_2}^{II}(f_1;s_1)
{\bf V}_{r-1}^I (g_2,f_3,\ldots,f_r; \ s_1+s_2,s_3,\ldots,s_r)
+{\bf Z}^I (f_1,\ldots,f_r;\ s_1,\ldots,s_r),
\end{split}
\end{equation}
where
\begin{equation}\label{zbav}
{\bf Z}^I (f_1,\ldots,f_r;\ s_1,\ldots,s_r)
=  -  {\bf B_2}^{II}(f_1;s_1) {\bf A_2}^{IJ}(f_1;s_1)
{\bf V}_r^J(f_1,\ldots,f_r;\ s_1,\ldots,s_r).
\end{equation}
All the series of meromorphic functions involved in the products of
matrices in formulas \eqref{vbvz} and \eqref{zbav} converge normally
on all compact subsets of $\C^r$. Moreover, all the entries of the
matrices on the right hand side of \eqref{zbav} are holomorphic on
the open set $U_r(q)$. Therefore
the entries of ${\bf Z}^I (f_1,\ldots,f_r;\ s_1,\ldots,s_r)$ are also
holomorphic in $U_r(q)$.

If we write $\pi_q(f_1,\ldots,f_r;\ s_1,\ldots,s_r)$ to be the first entry
of ${\bf Z}^I(f_1,\ldots,f_r;\ s_1,\ldots,s_r)$,
we then get from \eqref{vbvz} that
\begin{equation}\label{explicit-malf-2}
\begin{split}
&L_r(f_1,\ldots,f_r;\ s_1,\ldots,s_r) \\
& =f_1(1) \ L_{r-1}(g_2,f_3,\ldots,f_r;\ s_1+s_2,s_3,\ldots,s_r)\\
& \ \ + f_1(1) \sum_{k=1}^{q-1} (s_1)_k \frac{A_k(f(1))}{(f(1)-1)^k} \
L_{r-1}(g_2,f_3,\ldots,f_r;\ s_1+s_2+k,s_3,\ldots,s_r)\\
& \ \ + \pi_q(f_1,\ldots,f_r;\ s_1,\ldots,s_r),
\end{split}
\end{equation}
where $\pi_q(f_1,\ldots,f_r;\ s_1,\ldots,s_r)$ is holomorphic 
in the open set $U_r(q)$.

\section{Poles and residues}

In this section, using the formulas
\eqref{vbvw-malf}, \eqref{vbvz} and 
the induction on depth $r$,
we obtain a list of possible singularities
of the multiple additive $L$-functions.
From \rmkref{anacont-alf}, we know that
when $r=1$ and $f_1(1)=1$, then
$L_1(f_1;s_1)$ can be extended to a meromorphic function
with only a simple pole at $s_1=1$ with residue $1$,
and if $r=1$ but $f_1(1) \not=1$, then $L_1(f_1;s_1)$
can be extended to an entire function.

First we obtain an expression for the residues along the possible
polar hyperplane of the multiple additive $L$-functions
and then we deduce the exact set of singularities of these functions.
A theorem of G. Frobenius \cite{GF} about the zeros of
Eulerian polynomials plays a crucial role in this analysis.

\subsection{Set of all possible singularities}

The following theorem gives a description of 
possible singularities of the multiple additive $L$-functions.

\begin{thm}\label{poles-malf}
The multiple additive $L$-function $L_r(f_1,\ldots,f_r;\ s_1,\ldots,s_r)$
has different set of singularities depending on 
the values of  $g_i(1)$ for $1 \le i \le r$.

$(a)$ If $g_i(1) \neq 1$ for all $1 \le i \le r$, then
$L_r(f_1,\ldots,f_r;\ s_1,\ldots,s_r)$ is a holomorphic function on $\C^r$.

Now let $i_1 < \cdots < i_m$ be all the indices such that $g_{i_j}(1) = 1$
for all $1 \le j \le m $. Then the set of all possible singularities of
$L_r(f_1,\ldots,f_r;\ s_1,\ldots,s_r)$ is described as follows:

$(b)$ If $i_1=1$, then $L_r(f_1,\ldots,f_r; s_1,\ldots,s_r)$ is
holomorphic outside the union of hyperplanes given by the equations
$$
s_1=1; ~s_1 + \cdots+ s_{i_j}=n \ \text{ for all } n \in \Z_{\le j} \ \text{and }
2 \le j \le m.
$$
It has at most simple poles along each of these hyperplanes.

$(c)$ If $i_1 \not =1$, then $L_r(f_1,\ldots,f_r; s_1,\ldots,s_r)$ is
holomorphic outside the union of hyperplanes given by the equations
$$
s_1 + \cdots+ s_{i_j}=n \text{ for all } n \in \Z_{\le j} \text{  and }
1 \le j \le m.
$$
It has at most simple poles along each of these hyperplanes.
\end{thm}

\begin{proof}
We prove the assertions $(a),(b)$ and $(c)$ separately.
 
{\it Proof of $(a)$:} From \rmkref{anacont-alf}, we know
that the assertion $(a)$ is
true for depth $1$ multiple additive $L$-function. 
Now let $q \ge 1$ be an integer and $I=I_q, ~J = J_q$
be as before. We will make use of 
equation \eqref{vbvz} for our proof.

The entries of the first row of the matrix ${\bf B_2}^{II}(f_1;s_1)$
are holomorphic on $\C^r$ and by the induction hypothesis,
the entries of the column matrix
${\bf V}_{r-1}^I (g_2,f_3,\ldots,f_r; \ s_1+s_2,s_3,\ldots,s_r)$
are holomorphic on $\C^r$. Further the entries of the column vector
${\bf Z}^I (f_1,\ldots,f_r;\ s_1,\ldots,s_r)$ are holomorphic in $U_r(q)$.
Since the open sets $U_r(q)$ for $q \ge 1$ cover $\C^r$, assertion
$(a)$ follows.\\

{\it Proof of $(b)$:} For depth $1$ multiple additive $L$-function,
the assertion $(b)$ follows from \rmkref{anacont-alf}.
Again let $q \ge 1$ be an integer and $I=I_q, ~J = J_q$
be as defined earlier. We now complete the proof of
assertion $(b)$ by making use of the equation \eqref{vbvw-malf}.

The entries of the first row of the matrix ${\bf B_1}^{II}(s_1-1)$
are holomorphic outside the hyperplane given by the equation
$s_1=1$ and have at most simple pole along this hyperplane.
By the induction hypothesis, the entries of the column vector
${\bf V}_{r-1}^I(f_2,\ldots,f_r;\ s_1+s_2-1,s_3,\ldots,s_r)$
are holomorphic outside the union of the hyperplanes
given by the equations
$$
s_1 + \cdots+ s_{i_j}=n \ \text{ for all } n \in \Z_{\le j}, \ \text{ for all }
2 \le j \le m
$$
and have at most simple poles
along these hyperplanes. Finally the entries of the column vector
${\bf W}^I(f_1,\ldots,f_r;\ s_1,\ldots,s_r)$ are holomorphic in $U_r(q)$.
Since $\C^r$ is covered by open sets of the form
$U_r(q)$ for $q \ge 1$, assertion $(b)$ follows.\\

{\it Proof of $(c)$:} The proof of this assertion follows along the lines
of the assertion $(a)$. The only difference is the induction hypothesis. 
Here the induction hypothesis implies that the entries of the
column matrix
${\bf V}_{r-1}^I (g_2,f_3,\ldots,f_r; \ s_1+s_2,s_3,\ldots,s_r)$
are holomorphic outside the union of the hyperplanes
given by the equations
$$
s_1 + \cdots+ s_{i_j}=n \ \text{ for all } n \in \Z_{\le j}, \ \text{ for all }
1 \le j \le m.
$$
Now the proof of assertion $(c)$ follows 
mutatis mutandis the proof of assertion $(a)$. 
\end{proof}

\subsection{Expression for residues}

Here we compute the residues of the multiple additive $L$-function
of depth $r$ along the possible polar hyperplanes.

\begin{thm}\label{residues-malf}
Let $i_1 < \cdots < i_m$ be the indices such that $g_{i_j}(1) = 1$
for all $1 \le j \le m $. If $i_1=1$, then $L_r(f_1,\ldots,f_r;\ s_1,\ldots,s_r)$
has a polar singularity along the hyperplane given by the equation $s_1=1$
and the residue is the restriction of $L_{r-1}(f_2,\ldots,f_r;\ s_2,\ldots,s_r)$
to this hyperplane. 

In general, the residue along the hyperplane 
given by the equation
$$
s_1 + \cdots+ s_{i_j}=n \ \text{ for } n \in \Z_{\le j} \ \text{ and }
1 \le j \le m
$$
is the restriction 
of the product of $L_{r-i_j}(f_{i_j +1},\ldots,f_r;\ s_{i_j +1},\ldots,s_r)$ with
$(0,j-n)$-th entry of 
\begin{equation*}
\begin{split}
{\bf C}_j := &\left( \prod_{i=1}^{i_1-1} g_i(1) {\bf B_2}(g_i;s_1+\cdots+s_i)  \right)\\
& \times \prod_{k=1}^{j-1} \left( {\bf B_1}(s_1+\cdots+s_{i_k}-k)
\prod_{i=i_k+1}^{i_{k+1}-1}  g_i(1) {\bf B_2}(g_i;s_1+\cdots+s_i - k) \right)
\end{split}
\end{equation*}
 to the above hyperplane.
\end{thm}

\begin{proof}
First suppose that $i_1=1$. Then for any integer $q \ge 1$, we deduce from
\eqref{explicit-malf-1} and \thmref{poles-malf} that $$
L_r(f_1,\ldots,f_r;\ s_1,\ldots,s_r) -
\frac{1}{s_1-1} \ L_{r-1}(f_2,\ldots,f_r;\ s_1+s_2-1,s_3,\ldots,s_r)
$$
has no pole inside the open set $U_r(q)$
along the hyperplane given by the equation $s_1 =~1$. 
These open sets cover $\C^r$ and hence the residue of 
$L_r(f_1,\ldots,f_r;\ s_1,\ldots,s_r)$
along the hyperplane given by the equation $s_1=1$ is the restriction
of the meromorphic function 
$L_{r-1}(f_2,\ldots,f_r;\ s_2,\ldots,s_r)$ to
the hyperplane given by the equation $s_1=1$.
This proves the first part of \thmref{residues-malf}.

Next let $q \ge j-n+1$ be an integer and $I=I_q, ~J = J_q$
be as before. Now to determine the residue along the hyperplane
$$
s_1 + \cdots+ s_{i_j}=n \ \text{ for } n \in \Z_{\le j} \ \text{ and }
1 \le j \le m,
$$
we iterate the formulas \eqref{vbvw-malf} and \eqref{vbvz}
according to the applicable case and obtain that
\begin{equation}\label{vcvz}
\begin{split}
&{\bf V}_r^I(f_1,\ldots,f_r;\ s_1,\ldots,s_r)\\
&={\bf C}_j ^{II}
{\bf V}_{r-i_j+1}^I (g_{i_j},f_{i_j +1},\ldots,f_r; \
s_1+\cdots+s_{i_j}-(j-1),s_{i_j +1},\ldots,s_r)\\
& \ \ \ +{\bf Z}^{j,I} (f_1,\ldots,f_r;\ s_1,\ldots,s_r).
\end{split}
\end{equation}
Here ${\bf Z}^{j,I} (f_1,\ldots,f_r;\ s_1,\ldots,s_r)$ is a column
matrix whose entries are finite
sums of products of rational functions in $s_1,\ldots,s_{i_j-1}$ with meromorphic
functions which are holomorphic in $U_r(q)$. 
These entries therefore have no
singularity along the hyperplane given by the equation
$s_1 + \cdots+ s_{i_j}=n$ in $U_r(q)$.
The entries of ${\bf C}_j ^{II}$ 
are rational functions in $s_1,\ldots,s_{i_j-1}$ and hence
again have no singularity along the above hyperplane.
It now follows from the first part of \thmref{residues-malf} that the only entry 
of
$$
{\bf V}_{r-i_j+1}^I (g_{i_j},f_{i_j +1},\ldots,f_r; \
s_1+\cdots+s_{i_j}-(j-1),s_{i_j +1},\ldots,s_r)
$$
that can possibly have a pole in $U_r(q)$
along the hyperplane given by the equation
$s_1 + \cdots+ s_{i_j}=n$  is
the one of index $j-n$, which is
$$
L_{r-i_j+1} (g_{i_j},f_{i_j +1},\ldots,f_r; \
s_1+\cdots+s_{i_j}-n+1,s_{i_j +1},\ldots,s_r).
$$
Again by the first part of \thmref{residues-malf}, the residue
of this function along the hyperplane given by the equation
$s_1 + \cdots+ s_{i_j}=n$ is the restriction of
$L_{r-i_j} (f_{i_j +1},\ldots,f_r;\ s_{i_j +1},\ldots,s_r)$ to
this hyperplane. Thus the residue of $L_r(f_1,\ldots,f_r;\ s_1,\ldots,s_r)$
along the hyperplane given by the equation $s_1 + \cdots+ s_{i_j}=n$
is the product of $L_{r-i_j}(f_{i_j +1},\ldots,f_r;\ s_{i_j +1},\ldots,s_r)$ with
$(0,j-n)$-th entry of ${\bf C}_j ^{II}$. Since open sets of the form
$U_r(q)$ cover $\C^r$, this completes the proof of \thmref{residues-malf}.
\end{proof}

\subsection{Exact set of singularities}

Now we will deduce the exact set of singularities of the multiple
additive $L$-functions. As mentioned earlier, a theorem of
Frobenius \cite{GF} about the zeros of Eulerian polynomials
helps us in determining
the exact set of singularities. We recall that the
Eulerian polynomials $A_n(t)$'s are defined by the following exponential
generating function
$$
\sum_{n \ge 0} A_n(t) \frac{y^n}{n!}
= \frac{1-t}{e^{(t-1)y}-t}.
$$
These polynomials satisfy the following recurrence relation
$$
A_{0}(t) = 1 \ \text{ and } \ A_{n}(t) = t(1-t)A_{n-1}'(t) + A_{n-1}(t)(1+(n-1)t)
\ \text{ for all } \ n \ge 1. 
$$
Frobenius \cite{GF} proved the following theorem 
about the zeros of the Eulerian polynomials~$A_n(t)$.

\begin{thm}[Frobenius]\label{Frob}
All the zeros of the Eulerian polynomials $A_n(t)$ are real, negative and simple.
\end{thm}

From this theorem and the above recurrence formula, we can now deduce
the following corollary.

\begin{cor}\label{common-zero}
For any $n \ge 0$, $A_{n+1}(t)$ and $A_n(t)$ do not have a common zero.
\end{cor}

\begin{proof}
To see this, let $a$ be a zero of $A_n(t)$. Then by \thmref{Frob},
$A_n'(a) \neq 0$. Again by \thmref{Frob}, $a \neq 0,1$. Using
the above recurrence formula, we deduce that $A_{n+1}(a) \neq~0$.
\end{proof}

Recall that for an indeterminate $t$ and a non-trivial 
additive character $f:\Z \to \C^*$, we have

$$
{\bf B_1}(t) = \left( \begin{array}{c c c c c}
\frac{1}{t} & \frac{B_1}{1!} &  \frac{(t+1)B_2}{2!} &
\frac{(t+1)(t+2)B_3}{3!} & \cdots \\
0 & \frac{1}{t+1} & \frac{B_1}{1!} & \frac{(t+2)B_2}{2!} & \cdots \\
0 & 0 & \frac{1}{t+2} & \frac{B_1}{1!} & \cdots \\
0 & 0 & 0 & \frac{1}{t+3} & \cdots\\
\vdots & \vdots & \vdots & \vdots & \ddots
\end{array} \right)
$$
and
$$
{\bf B_2}(f;t) = \frac{1}{1-f(1)}
 \left( \begin{array}{c c c c c}
1 & \frac{1}{f(1)-1}t
& \frac{A_2(f(1))}{(f(1)-1)^2} \frac{t(t+1)}{2!}
& \frac{A_3(f(1))}{(f(1)-1)^3}\frac{t(t+1)(t+2)}{3!}
& \cdots \\
0 & 1 & \frac{1}{f(1)-1} (t+1)
& \frac{A_2(f(1))}{(f(1)-1)^2} \frac{(t+1)(t+2)}{2!} & \cdots \\
0 & 0 & 1 & \frac{1}{f(1)-1} (t+2) & \cdots \\
0 & 0 & 0 & 1 & \cdots\\
\vdots & \vdots & \vdots & \vdots & \ddots
\end{array} \right).
$$
Clearly non-zero elements of each row and each column
of these matrices are linearly independent as elements of
$\C(t)$. Also the first two entries of the first row of these 
matrices are non-zero.  Since we know that the
Bernoulli numbers 
\begin{equation}\label{Ber}
B_n~=~0  ~\iff~  n \text{  is odd and  } n \geq 3,
\end{equation}
we get that at least one of the first
two entries in every column of ${\bf B_1}(t)$ is non-zero. On the other
hand,  \corref{common-zero} implies that at least one of the first
two entries in every column of ${\bf B_2}(f;t)$ is non-zero.
Hence we know that for any two indeterminates $t_1,t_2$ and
any two non-trivial group homomorphisms $f_1,f_2:\Z \to \C^*$,
all entries of the first row of the matrices
${\bf B_1}(t_1) \ {\bf B_1}(t_2)$,
${\bf B_1}(t_1) \ {\bf B_2}(f_2;t_2)$, ${\bf B_2}(f_1;t_1)\ {\bf B_1}(t_2)$
and ${\bf B_2}(f_1;t_1)\ {\bf B_2}(f_2;t_2)$
are non-zero.

With this observation in place, we are now ready to 
determine the exact set of
singularities of the multiple additive $L$-functions.

\begin{thm}\label{exactpoles-malf}
The exact set of polar singularities of the multiple additive 
$L$-function\linebreak
$L_r(f_1, \ldots, f_r; \ s_1, \ldots, s_r)$ of depth $r$
differs from the set of all possible singularities (as listed
in \thmref{poles-malf}) only in the 
following two cases. Here we 
keep the notations of \thmref{poles-malf}.

a) If $i_1=1$ and $i_2=2$ i.e. $f_1(1)= f_2(1) =1$,
then $L_r(f_1,\ldots,f_r; s_1,\ldots,s_r)$
is holomorphic outside the union of hyperplanes given by the equations
\begin{align*}
& s_1=1; ~s_1+s_2=n \text{ for all } n \in \Z_{\le 2} \setminus J; \\
& s_1 + \cdots+ s_{i_j}=n \text{ for all } n \in \Z_{\le j},~ 
3 \le j \le m,
\end{align*}
where $J:=\{-2n-1:n\in \N\}$.
It has simple poles along each of these hyperplanes.

b) If $i_1=2$, then $L_r(f_1,\ldots,f_r; s_1,\ldots,s_r)$
is holomorphic outside the union of hyperplanes given by the equations
\begin{align*}
& s_1+s_2=n \text{ for all } n \in \Z_{\le 1} \setminus J; \\
& s_1 + \cdots+ s_{i_j}=n \text{ for all } n \in \Z_{\le j},~
2 \le j \le m,
\end{align*}
where $J:=\{1- n : n  \in I \}$ and $I:=\{ n  \in \N : A_{n}(f_1(1))=0\}$.
It has simple poles along each of these hyperplanes.
\end{thm}

\begin{proof}
When $1 \le j \le m$ and $n \in \Z_{\le j}$, the restriction of
$L_{r-i_j} (f_{i_j +1},\ldots,f_r;\ s_{i_j +1},\ldots,s_r)$
to the hyperplane given by the equation
$s_1 + \cdots+ s_{i_j}=n$
is a non-zero meromorphic function.

First suppose that $i_1=1$ and $i_2=2$. Then by \eqref{Ber},
we deduce that only non-zero entries in 
the first row of ${\bf C}_2$ are  of index $(0,1)$ and of index $(0, 2n)$ 
for $n \in \N$. Also we know that all the entries in the first
row of ${\bf C}_j$ for $3 \le j \le m$ are non-zero. Now
from \thmref{residues-malf}, we  conclude
that the exact set of singularities in this case consists of the hyperplanes
given by the equations
\begin{align*}
& s_1=1; ~s_1+s_2=n \text{ for all } n \in \Z_{\le 2} \setminus J; \\
& s_1 + \cdots+ s_{i_j}=n \text{ for all } n \in \Z_{\le j}, \text{ for all }
3 \le j \le m,
\end{align*}
where $J:=\{-2n-1:n\in \N\}$. This completes the proof of assertion $(a)$.

Next let $i_1=2$. Clearly the entries in the first row of ${\bf C}_1$ 
that are zero are of index $(0, n)$ for all $n \in I$. Also we know that
all entries in the first row of ${\bf C}_j$ for $2 \le j \le m$ are non-zero. 
Hence using \thmref{residues-malf}, we have that 
the exact set of singularities in this case consists of the hyperplanes
given by the equations
\begin{align*}
& s_1+s_2=n \text{ for all } n \in \Z_{\le 1} \setminus J; \\
& s_1 + \cdots+ s_{i_j}=n \text{ for all } n \in \Z_{\le j}, \text{ for all }
2 \le j \le m,
\end{align*}
where $J:=\{1- n :  n \in I \}$.

Now in all other cases, applying \thmref{poles-malf}, we see that
the hyperplanes
given by the equations of the form $s_1+s_2=n$ are not singularities
of $L_r(f_1, \ldots, f_r; \ s_1, \ldots, s_r)$. 
Using \thmref{residues-malf}, we know that the expression
for residues along
the possible polar hyperplanes given by the equations of the form
$s_1+ \cdots+ s_{i_j}=n$ involves product of at least two matrices 
of the form ${\bf B_1}(x)$ and ${\bf B_2}(f;y)$ where $x$ and $y$ 
are two different indeterminates. Hence such an expression is non-zero.
This completes the proof of \thmref{exactpoles-malf}.
\end{proof}

\begin{ex}
\rm We know that $t=-1$ is a zero for the Eulerian polynomials
$A_{n}(t)$ only when $n$ is even. Suppose that we are in a case
when $f_1(1)= -1$ and $i_1=2$ i.e. $f_1(1)= f_2(1) = -1$.
Now \thmref{exactpoles-malf} implies that the 
exact set of singularities
of $L_r(f_1, \ldots, f_r; \ s_1, \ldots, s_r)$ 
consists of the hyperplanes given by the equation
\begin{align*}
& s_1+s_2=n \text{ for all } n \in \Z_{\le 1} \setminus J; \\
& s_1 + \cdots+ s_{i_j}=n \text{ for all } n \in \Z_{\le j},~
2 \le j \le m,
\end{align*}
where $J:=\{-2n-1:n\in \N\}$. Further 
$L_r(f_1, \ldots, f_r; \ s_1, \ldots, s_r)$
has simple poles along these hyperplanes.
\end{ex}

\section{Concluding remarks}

Though at this moment it seems difficult to determine the
exact set of singularities of the multiple Dirichlet $L$-functions,
we can still extend the theorem of Akiyama and Ishikawa
for Dirichlet characters, not necessarily of same modulus.
As an immediate consequence of \eqref{L-Psi-multiple}, \thmref{anacont-malf}
and \thmref{poles-malf}, we derive the following theorem.

\begin{thm}\label{poles-mdlf}
Let $r \ge 1$ be an integer and $\chi_1,\ldots,\chi_r$ be Dirichlet characters
of arbitrary modulus. Then the multiple Dirichlet $L$-function
$L_r(s_1,\ldots,s_r;~\chi_1,\ldots,\chi_r)$
of depth $r$ can be extended as a
meromorphic function to $\C^r$ with possible simple poles at
the hyperplanes given by the equations
$$
s_1=1 ; ~s_1+\cdots+s_i =n  \ \text{ for all } \ n \in \Z_{\le i}, ~ 2 \le i \le r.
$$
\end{thm}

It has been discussed in detail in \cite{BS-th} that how one could
deduce translation formulas and thereafter obtain meromorphic continuation
and information about singularities
for various Dirichlet series and their several variable counterparts.
In view of this, we conclude our article by the following question.
An answer to this will perhaps help us to determine the
exact set of singularities of the multiple Dirichlet $L$-functions.

\begin{ques}
Is it possible to derive a translation formula satisfied by the
multiple Dirichlet $L$-functions of depth $r$, which involves
its translates with respect to the first variable $s_1$ and also
involves the multiple Dirichlet $L$-functions of depth $r-1$?
\end{ques}

{\bf Acknowledgement:}
I would like to thank Prof. Joseph Oesterl\'e for his helpful
suggestions and comments.

\end{document}